\documentclass[11pt,reqno]{amsart}

 \usepackage[a4paper,twoside,top=2.5cm, bottom=2.5cm,left=2.5cm, right=2.5cm]{geometry}
 \usepackage{amsthm}
 \usepackage{amsmath}
 \usepackage{amsfonts}
 \usepackage{amssymb}
 \usepackage{amscd}
 \usepackage[mathscr]{eucal}
 \usepackage{lmodern}% http://tex.stackexchange.com/questions/58087/how-to-remove-the-warnings-font-shape-ot1-cmss-m-n-in-size-4-not-available
 \usepackage{tipa}% to use\texthtk
 \usepackage{bold-extra}
 \usepackage{textcomp}% this gives \textreferencemark
 \usepackage{graphicx}
 \usepackage[dvipsnames]{xcolor}
 \usepackage{enumitem}
 \usepackage{tikz} % for commutative diagram, graphs, matrices of symbols etc.
\usepackage{wasysym}
 \usepackage{fancybox}
 \usepackage{fancyhdr}
 \usepackage[nodisplayskipstretch]{setspace}
 \usepackage{hyperref}

 \newtheoremstyle{mydefinition}{20pt}{20pt}{}{}{\sffamily}{.}{.5em}{}
 \theoremstyle{mydefinition}
  \newtheorem{defn}{Definition}[section]
  \newtheorem{eg}[defn]{Example}

 \newtheoremstyle{myplain}{20pt}{20pt}{\itshape}{}{\sffamily}{.}{.5em}{}
 \theoremstyle{myplain}
  \newtheorem{thm}[defn]{Theorem}
  \newtheorem{lem}[defn]{Lemma}
  
  \newtheorem{cor}[defn]{Corollary}

 \newtheoremstyle{myremark}{20pt}{20pt}{}{}{\sffamily}{.}{.5em}{}
 \theoremstyle{myremark}

  \newtheorem{rmk}[defn]{Remark}

  \newtheorem*{ackn}{Acknowledgment}
 \renewcommand{\bf}[1]{\textbf{#1}}
 
 \renewcommand{\sc}[1]{\textsc{#1}}
 \renewcommand{\sf}[1]{\textsf{#1}}
 
 \renewcommand{\tt}[1]{\texttt{#1}}

 \newcommand{\mbb}[1]{\mathbb{#1}}
 \newcommand{\mcl}[1]{\mathcal{#1}}

 \newcommand{\msc}[1]{\mathscr{#1}}

 \newcommand{\ol}[1]{\overline{#1}}
 
 \newcommand{\wtilde}[1]{\widetilde{#1}}
 \newcommand{\what}[1]{\widehat{#1}}

 \newcommand{\abs}[1]{\left\lvert#1\right\rvert}

 \newcommand{\norm}[1]{\left\lVert#1\right\rVert}
 \newcommand{\bnorm}[1]{\bigl\lVert#1\bigr\rVert}
 
 \newcommand{\snorm}[1]{\norm{\smash{#1}}}

 \newcommand{\B}[1]{\msc{B}({#1})}

 \newcommand{\ip}[1]{\langle#1\rangle}
 \newcommand{\bip}[1]{\bigl\langle#1\bigr\rangle}

 \newcommand{\clran}[1]{\ol{\sf{ran}}(#1)}

 \newcommand{\mscriptsize}[1]{{\setlength{\arraycolsep}{.3ex}\text{\scriptsize$#1$}}}

 \newcommand{\Matrix}[1]{\begin{bmatrix}#1\end{bmatrix}}

 \newcommand{\sMatrix}[1]{\mscriptsize{\Matrix{#1}}}

 \DeclareMathOperator{\lspan}{\sf{span}}
 \DeclareMathOperator{\cspan}{\ol{\lspan}}

 \numberwithin{equation}{section}%Code for numbering equations sectionwise
 \allowdisplaybreaks[4] %If you prefer a strategy of letting page breaks fall where they may, even in the middle of a multi-line equation, then you might put \allowdisplaybreaks[1] in the preamble of your document. An optional argument 1–4 can be used for finer control: [1] means allow page breaks, but avoid them as much as possible; values of 2,3,4 mean increasing permissiveness. When display breaks are enabled with \allowdisplaybreaks, the \\* command can be used to prohibit a pagebreak after a given line, as usual.
 \setlist[enumerate]{font=\upshape,noitemsep, topsep=0pt} % while enumerating the numbering will be in up-shape. Default is italics. Also reduce item seperation space.
 \setlist[itemize]{noitemsep, topsep=0pt}

\begin{document}
\title{On a  generalization of Ando's dilation theorem}

\author{Nirupama Mallick}
\thanks{\sc{Nirupama Mallick}, Chennai Mathematical Institute, H1, SIPCOT IT Park, Kelambakkam, Siruseri, Tamilnadu- 603103, India. Email: \tt{niru.mallick@gmail.com}}

\author{K. Sumesh}
\thanks{\sc{K. Sumesh}, Indian Institute of Technology Madras, Sardar Patel Road, Opposite to C. L.R.I, Adyar, Chennai-600036, India. Email: {\tt{sumeshkpl@gmail.com, sumeshkpl@iitm.ac.in}}}

\date{\today}


\maketitle

\begin{abstract}
 We introduce the notion of $Q$-commuting operators which includes  commuting operators. We prove a generalized version of the commutant lifting theorem and Ando's dilation theorem in the context of $Q$-commuting operators. 
 
%\noindent \textbf{MSC (2010):}  {47A20.}
%
%\noindent \textbf{Key words:} Isometric dilation, unitary dilation, co-isometric extension, commutant lifting, Ando's dilation theorem.
\end{abstract}

\section{Introduction}

 Throughout, $\mcl{H}$ and $\mcl{K}$ denote complex Hilbert spaces and $\B{\cdot}$ denotes the space of all bounded linear maps.  Suppose $\mcl{H}_i$ and $\mcl{K}_i$ are Hilbert spaces such that  $\mcl{H}_i\subseteq\mcl{K}_i,~i=1,2$. Given $T\in\B{\mcl{H}_1,\mcl{H}_2}$ a bounded linear operator $S\in\B{\mcl{K}_1, \mcl{K}_2}$ is said to be an 
 \begin{enumerate}[label=(\roman*)]
     \item \emph{extension} of $T$ if $S(h)=T(h)$ for all $h\in\mcl{H}_1$. (In such cases we write $S|_{\mcl{H}_1}=T$.)  
     \item \emph{lifting} of $T$ if $S(\mcl{H}_1^\perp)\subseteq\mcl{H}_2^\perp$ and $T=P_{\mcl{H}_2}S|_{\mcl{H}_1}$ (equivalently $S^*|_{\mcl{H}_2}=T^*$).
 \end{enumerate}
 With respect to the decompositions $\mcl{K}_1=\mcl{H}_1\oplus\mcl{H}_1^\perp$ and $\mcl{K}_2=\mcl{H}_2\oplus\mcl{H}_2^\perp$, in $(i)$ and $(ii)$ the operator $S$ has the matrix form  $S=\sMatrix{T&\ast\\ 0&\ast}$ and $S=\sMatrix{T& 0\\ \ast&\ast}$, respectively. Note that $S$ is a lifting of $T$ if and only if $S^*$ is an extension of $T^*$. 
 
 An operator $S\in\B{\mcl{K}}$ is said to be a \emph{dilation} of $T\in\B{\mcl{H}}$ if $\mcl{H}\subseteq\mcl{K}$ and $T^n=P_{\mcl{H}}S^n|_{\mcl{H}}$ for all $n\geq 0$, where $P_{\mcl{H}}\in\B{\mcl{K}}$ is the orthogonal projection onto $\mcl{H}$. In such case, with respect to the decomposition $\mcl{K}=\mcl{H}\oplus\mcl{H}^\perp$, the operator  $S^n$ has the matrix form $S^n=\sMatrix{T^n&\ast\\ \ast&\ast}$ for all $n\geq 0$.  Clearly extension and lifting of $T\in\B{\mcl{H}}$ are dilations.  It is well known that (see \cite{NaFo68,NaFo70}) given any contraction $T\in\B{\mcl{H}}$ there exists a Hilbert space $\mcl{K}\supseteq\mcl{H}$ and an isometry $V\in\B{\mcl{K}}$ such that $V$ is a dilation of  $T$. Such a pair $(V,\mcl{K})$ is called an \emph{isometric dilation} of $T$. An isometric dilation $(V,\mcl{K})$ is said to be \emph{minimal} if 
 \begin{align}\label{eq-MID}
      \mcl{K}=\cspan\{V^n(\mcl{H}): n\geq 0\}.
 \end{align}
 Minimal isometric dilations are unique upto unitary equivalence in the sense that if $(V,\mcl{K})$ and  $(V',\mcl{K}')$ are two minimal isometric dilations, then there exists a unitary $U:\mcl{K}\to\mcl{K}'$ such that $U|_{\mcl{H}}=I_{\mcl{H}}$ and $UV=V'U$.  Sz-Nagy proved (\cite{Nagy53,NaFo70}) that given a contraction $T\in\B{\mcl{H}}$ there exists a Hilbert space $\mcl{K}$ and a unitary $U\in\B{\mcl{K}}$ such that $U$ is a dilation of $T$. Such a pair $(U,\mcl{K})$ is called a \emph{unitary dilation} of $T$. Moreover, such a dilation is unique (up to unitary equivalence) if it is \emph{minimal}, in the sense that 
 \begin{align}\label{eq-MUD}
      \mcl{K}=\cspan\{U^n(\mcl{H}): n\in\mbb{Z}\}.
 \end{align} 
 In \cite{Sch55} Schaffer  gave an elementary proof of the existence of minimal unitary dilation of a contraction.
 
 Given a contraction $T\in\B{\mcl{H}}$ there always exists a Hilbert space $\mcl{K}\supseteq\mcl{H}$ and a co-isometry $W\in\B{\mcl{K}}$ which extends $T$, and we call $(W,\mcl{K})$ a \emph{co-isometric extension} of $T$. In fact, by considering the lower right-hand corner of the matrix form of the Schaffer's construction (see \cite{Sch55,NaFo70}),  one can get a co-isometric extension $(W,\mcl{K})$ of $T$  which is \emph{minimal} in the sense that  
 \begin{align}
   \mcl{K}=\cspan\{W^{*n}(\mcl{H}): n\geq 0\}.  
 \end{align}
 Note that if $(V^*,\mcl{K})$ is a minimal co-isometric extension of $T^*$, then $(V,\mcl{K})$ is an isometric lifting and hence a dilation of $T$ which is \emph{minimal} in the sense that \eqref{eq-MID} holds. Now from the uniqueness property, it follows that $\mcl{\mcl{H}^\perp}$ is invariant for every minimal isometric dilation of $T\in\B{\mcl{H}}$. Thus, a pair $(V,\mcl{K})$ is a minimal isometric dilation of $T$ if and only if $(V,\mcl{K})$ is a minimal isometric lifting of $T$ if and only if $(V^*,\mcl{K})$ is a minimal co-isometric extension of $T^*$. 
 
 Ando (\cite{Ando63}) proved that given any two  contractions $T_1,T_2\in\B{\mcl{H}}$ which are commuting (i.e., $T_1T_2=T_2T_1$) there exists a Hilbert space $\mcl{K}_0\supseteq\mcl{H}$ and commuting isometries $V_1,V_2\in\B{\mcl{K}_0}$ such that $$T_1^nT_2^m=P_\mcl{H}V_1^nV_2^m|_{\mcl{H}}$$ for all $n,m\geq 0$. In fact, $V_i$ can be chosen to be a lifting of $T_i, i=1,2$. Further, using Ito's theorem \cite{Ito58} he concluded that there exists a Hilbert space  $\mcl{K}\supseteq\mcl{H}$ and commuting unitary operators $U_1,U_2\in\B{\mcl{K}}$ such that $$T_1^nT_2^m=P_\mcl{H}U_1^nU_2^m|_{\mcl{H}}$$ for all $n,m\geq 0$. This is known as the \emph{Ando's dilation} theorem.
 
 Let $T_i\in\B{\mcl{H}_i}$ be a  contraction with isometric lifting $(V_i,\mcl{K}_i)$ (respectively co-isometric extension $(W_i,\mcl{K}_i))$, ~$i=1,2$.  Suppose $X\in\B{\mcl{H}_1,\mcl{H}_2}$  intertwines $T_1$ and $ T_2$, i.e., $XT_1=T_2X$. Then, due to Sz-Nagy and Foias (\cite{NaFo68},\cite{NaFo70}) there exists a norm-preserving lifting (respectively extension) $Y\in\B{\mcl{K}_1,\mcl{K}_2}$ of $X$ which intertwine $V_1$ and $V_2$ (respectively $W_1$ and $W_2$). This result is called \emph{intertwining lifting} (respectively \emph{intertwining co-extension}) theorem.  The case when $T_1=T_2$ and $V_1=V_2$ (respectively $W_1=W_2$) is known as \emph{commutant lifting} (respectively \emph{commutant co-extension}) theorem. 
 
 One may ask how  these dilation theorems of commuting pair of contractions can be  generalized to the setting of noncommuting pair of contractions $T_1,T_2$? In \cite{Seb94a} Sebestyen proved analogues of commutant lifting theorem and Ando's dilation theorem for \emph{anticommuting} pair (i.e., $T_2T_1=-T_1T_2$) of contractions. In \cite{DN18} Keshari and Mallick considered \emph{$q$-commuting} operators  (i.e., $T_2T_1=qT_1T_2$) where $q\in\mbb{T}=\{z\in\mbb{C}: \abs{z}=1\}$. They proved a generalized version of the commutant lifting theorem, intertwining lifting theorem and Ando's dilation theorem in the context of $q$-commuting contractions, and called them as \emph{$q$-commutant lifting} theorem, \emph{$q$-intertwining lifting} theorem and \emph{$q$-commutant dilation} theorem respectively. In this article we consider operators $T_1$ and $T_2$ which are \emph{$Q$-commuting} i.e.,  $T_2T_1$ equals either $QT_1T_2, T_1QT_2$ or $T_1T_2Q$ for some bounded operator $Q$. Our main aim is to prove an analogue of the commutant lifting theorem and Ando's dilation theorem  to the setting of $Q$-commuting operators.  As a first step we characterize (Theorem \ref{thm-CLT}) $Q$-commutants of a contraction $T\in\B{\mcl{H}}$ in terms of $\ol{Q}$-commutants of its minimal isometric dilation $(V,\mcl{K})$, where $\ol{Q}=Q\oplus Q'\in\B{\mcl{K}}$ with $Q'\in\B{\mcl{H}^\perp}$. This is a generalized version of   ($q$-)commutant lifting theorem (\cite{DN18,NaFo68}). The proof uses Schaffer construction. Further using ideas from \cite{DMP68} we prove generalized versions (see Theorem \ref{thm-CIL}, \ref{thm-unitary-CIT}) of ($q$-)intertwining lifting theorem. In Theorem \ref{thm-unitary-CDT} we characterize $Q$-commutants of a contraction in terms of $\ol{Q}$-commutants of its minimal unitary  dilation. Finally, we prove our main theorems that  $Q$-commuting contractions can be dilated into $\ol{Q}$-commuting isometries (Theorem \ref{thm-Q-CID}) and further into $\ol{Q}$-commuting unitaries (Theorem \ref{thm-main}). These results generalize Ando's dilation theorem and $q$-commutant dilation theorem.  The proofs consist of standard dilation theoretic arguments.

\section{Main results}

 Suppose $\mcl{H}$ and $\mcl{K}$ are Hilbert spaces such that $\mcl{H}\subseteq\mcl{K}$. Given any $Q\in\B{\mcl{H}}$ we let $\ol{Q}_{\mcl{K}}$ (or simply $\ol{Q}$) denotes any bounded operator on $\mcl{K}$ such that $\mcl{H}$ is a reducing subspace for $\ol{Q}$ and $\ol{Q}|_{\mcl{H}}=Q$. Note that $(Q\oplus q I_{\mcl{H}^\perp})\in\B{\mcl{K}}$ is an example for such an operator $\ol{Q}$ for every $q\in\mbb{C}$. If $Q$ is a contraction or (co-)isometry or unitary, then we require $\ol{Q}$ also to be a contraction or (co-)isometry or unitary, respectively.  
   
\begin{defn}\label{defn-Q-comm}
 Given $Q\in\B{\mcl{H}}$ two operators $T_1,T_2\in\B{\mcl{H}}$ are said to be \emph{$Q$-commuting} if one of the following happens: 
 \begin{align}\label{eq-Q-cmtng}
      T_2T_1=QT_1T_2
      \mbox{ or }
      T_2T_1=T_1QT_2
      \mbox{ or }
      T_2T_1=T_1T_2Q.
 \end{align} 
 If $Q=I_\mcl{H}$ (respectively $Q=-I_{\mcl{H}}$), then $Q$-commuting means commuting (respectively anti-commuting). 
\end{defn}
 
\begin{eg}
 Let $T_1=\sMatrix{0&0\\ 1&1}$ and $T_2=\sMatrix{-2&0\\1&1}$ in $M_2(\mbb{C})$. Note that $T_1,T_2$ are not commuting. In fact, there does not exists any $q\in\mbb{C}$ such that $T_2T_1=qT_1T_2$. But $Q=\sMatrix{-1&0\\0&1}$ and  $Q'=\sMatrix{0 & 0\\ 0& 1}$ in $M_2(\mbb{C})$ are such that  $T_2T_1=T_1T_2Q=T_1Q'T_2$. Note that there is no $Q\in M_2(\mbb{C})$ such that $T_2T_1=QT_1T_2$. 
\end{eg}

\begin{eg}
 Suppose $L,R,Q,Q'\in\B{\ell^2}$ are the linear operators given by 
 \begin{align*}
      L(x_1,x_2,x_3,\cdots) &:= (x_2, x_3,x_4,\cdots)\\
      R(x_1,x_2,x_3,\cdots) &:=(0,x_1,x_2,x_3,\cdots)\\
      Q(x_1,x_2,x_3,\cdots) &:= (0, x_2, x_3,x_4,\cdots)\\
      Q'(x_1,x_2,x_3,\cdots) &:= (0, 0, x_3,x_4,\cdots)
 \end{align*}
 Clearly $qLR\neq RL$ for all $q\in\mbb{C}$.  Note that $RL=QLR=LQ'R=LRQ$. 
\end{eg}

\begin{eg}
 Suppose $T_1=\sMatrix{0&1\\0&0}$ and $T_2=\sMatrix{0&0\\2&0}$ in $M_2(\mbb{C})$. Note that there does not exists any $Q\in M_2(\mbb{C})$ such that $T_1$ and $T_2$ are $Q$-commuting.  
\end{eg}

 The above example shows that given two operators $S,T\in\B{\mcl{H}}$, there may not exist always an operator $Q\in\B{\mcl{H}}$ such that $S$ and $T$ are $Q$-commuting. However, the next Lemma says that  given two operators $T,Q\in\B{\mcl{H}}$, under some suitable conditions, there always exists an operator  $S\in\B{\mcl{H}}$ such that $S$ and $T$ are $Q$-commuting. This is a generalization of \cite[Lemma 3.5]{DN18}. Recall that $T\in\B{\mcl{H}}$ is called a pure co-isometry if $TT^*=I$ and $T^{\ast n}\to 0$ in the strong operator topology.

\begin{lem}
 Suppose $T\in\B{\mcl{H}}$ is a pure co-isometry and $Q\in\B{\mcl{H}}$ is an isometry. If $T^*(\mcl{H})$ is invariant for $Q$, then there exists a co-isometry $S\in\B{\mcl{H}}$ such that $TS=STQ$.
\end{lem}

\begin{proof}
 Since $T^*$ is an isometry $\mcl{W}=(T^*\mcl{H})^\perp$ is a wandering subspace for $T^*$, i.e., $T^{\ast m}(\mcl{W})\perp T^{\ast n}(\mcl{W})$ for all $m\neq n\in\mbb{N}$. Moreover, since $T$ is pure co-isometry $\mcl{H}=\bigoplus_{n=0}^\infty T^{\ast n}(\mcl{W})$. As $T^*(\mcl{H})\subseteq\mcl{W}^\perp$  and $Q^*(\mcl{W})\subseteq\mcl{W}$, for $m<n$ and $w,w'\in\mcl{W}$ we have 
\begin{align*}
  \bip{(QT^*)^n w,(QT^*)^mw'}
     =\bip{(QT^*)^{n-m} w,w'}
     =\bip{T^*(QT^*)^{n-m-1} w,Q^*w'}
     =0.
\end{align*} 
Thus $(QT^*)^m(\mcl{W})\perp (QT^*)^n(\mcl{W})$ for all $m\neq n\in\mbb{N}\cup\{0\}$.  Define $S_0:\mcl{H}\to\mcl{H}$ by
  $$S_0(\sum_{n=0}^\infty T^{\ast n}w_n)=\sum_{n=0}^\infty (QT^*)^{n+1}w_n.$$ Then,  
 \begin{align*}
   \bnorm{S_0(\sum_{n\geq 0} T^{\ast n}w_n)}^2
        &=\sum_{n\geq 0}\sum_{m\geq 0}\bip{(QT^*)^{n+1}w_n,(QT^*)^{m+1}w_m} \\
        &=\sum_{n\geq 0}\bip{(QT^*)^{n+1}w_n,(QT^*)^{n+1}w_n} \\
        %&=\sum_{n\geq 0}\bip{T^*(QT^*)^{n}w_n,Q^*QT^*(QT^*)^{n}w_n} \\
        %&\leq\norm{Q}^2\sum_{n\geq 0}\bip{T^*(QT^*)^{n}w_n,T^*(QT^*)^{n}w_n} \\
       % &\leq\sum_{n\geq 0}\bip{(QT^*)^{n}w_n,(QT^*)^{n}w_n} \\
        %&~~\vdots\\
        &=\sum_{n\geq 0}\bip{w_n,w_n} \\   
        &=\sum_{n\geq 0}\bip{T^{*n}w_n,T^{*n}w_n} \\
        &=\sum_{n\geq 0}\sum_{m\geq 0}\bip{T^{*n}w_n,T^{*m}w_m} \\  
        &=\bnorm{\sum_{n\geq 0} T^{*n}w_n}^2. 
 \end{align*}
 Thus $S_0$ is a well-defined  isometry. Moreover, for $w_n\in\mcl{W},n\geq 1$ we have 
 \begin{align*}
     S_0T^*(\sum_{n\geq 0} T^{\ast n}w_n)
                %=S_0(\sum_{n\geq 0} T^{\ast (n+1)}w_n)
                =\sum_{n\geq 0} (QT^*)^{n+2}w_n
                %=QT^*\sum_{n\geq 0} (QT^*)^{n+1}w_n
                =QT^*S_0(\sum_{n\geq 0} T^{\ast n}w_n),
 \end{align*}
 hence $S_0T^*=QT^*S_0$. Take adjoint on both sides to get $TS=STQ$, where $S=S_0^*$. 
\end{proof}

\subsection{Lifting theorems}
 We recall some basic facts which we will be using frequently. Suppose $T\in\B{\mcl{H}}$ is a contraction with dilation $(V,\mcl{K})$ and let $Y\in\B{\mcl{K}}$ be an extension of $X\in\B{\mcl{H}}$. Then w.r.t to the decomposition $\mcl{K}=\mcl{H}\oplus\mcl{H}^\perp$ we have  $Y=\sMatrix{X&\ast\\ 0&\ast}$ and $V^n=\sMatrix{T^n & \ast\\ \ast &\ast}$ for all $n\geq 0$. Note that $V^nY^m=\sMatrix{T^nX^m&\ast\\ \ast&\ast}$, so that  $T^nX^m=P_{\mcl{H}}V^nY^m|_{\mcl{H}}$ for all $n,m\geq 0$. Similarly if $Y$ is any lifting of $X$, then  $X^nT^m=P_{\mcl{H}}Y^nV^m|_{\mcl{H}}$ for all $n,m\geq 0$.
 
 Now we prove an analogue of the ($q$-)commutant lifting theorem for $Q$-commuting operators. 
  
\begin{thm}[$Q$-commutant lifting]\label{thm-CLT}
 Let $T\in\B{\mcl{H}}$ be a contraction with isometric lifting $(V,\mcl{K})$, and let $X\in\B{\mcl{H}}$. Suppose $Q\in\B{\mcl{H}}$ and $\ol{Q}\in\B{\mcl{K}}$ are contractions.   
 \begin{enumerate}
   \item [(i)] 
          If $XT=QTX$, then there exists a lifting $Y\in\B{\mcl{K}}$ of $X$ such that $YV=\ol{Q}VY$.
   \item
          [(ii)] If $XT=TQX$, then there exists a lifting $Y\in\B{\mcl{K}}$ of $X$ such that $YV=V\ol{Q}Y$.
 \end{enumerate}    
 Further assume that $Q$ and $\ol{Q}$ are unitary. 
 \begin{enumerate}
  \item [(iii)]
         If $XT=TXQ$, then there exists a lifting $Y\in\B{\mcl{K}}$ of $X$ such that $YV=VY\ol{Q}$.  
 \end{enumerate}
 In all cases $T^nX^m=P_\mcl{H}V^nY^m|_\mcl{H}$ and $X^nT^m=P_\mcl{H}Y^nV^m|_\mcl{H}$ for all $n,m\geq 0$. Moreover,  $Y$ can be chosen such that $\norm{Y}=\norm{X}$. 
\end{thm}

\begin{proof} 
 (i)  Set $\what T=\sMatrix{QT&0\\0&T}$ and $\what {X}=\sMatrix{0&X\\0&0}$  on $\mcl{H}\oplus\mcl{H}$. Let $D=(I-\ol{Q}^*\ol{Q})^{\frac{1}{2}}\in\B{\mcl{K}}$ and $\mcl{K}_0=\clran{DV}\subseteq\mcl{K}$. Let $\wtilde{\mcl{K}}=\mcl{K}\bigoplus(\bigoplus_1^\infty \mcl{K}_0)$. We consider $\mcl{H}\subseteq\mcl{K}\subseteq\wtilde{\mcl{K}}$ through the canonical identification. Now define $\wtilde{V}\in\B{\wtilde{\mcl{K}}}$ by    
 \begin{align*}  
         \wtilde{V}=\Matrix{
       			 \ol{Q}V  &  0                   &             0 & \dots   & \cdots\\
       			 DV  &  0                   &              0 & \dots  & \cdots \\
       		         0     & I_{\mcl{K}_0} &              0 & \dots  & \cdots \\
       		         0     &                  0  & I_{\mcl {K}_0} & \dots &  \cdots\\
       			 \vdots& \vdots        & \vdots       &  \ddots    & \vdots         
       		    }.
 \end{align*} 
 Note that $\wtilde{V}$ is an isometry. Also since $\ol{Q}^*|_{\mcl{H}}=Q^*$ and $V^*|_{\mcl{H}}=T^*$ we have 
 \begin{align*}
    \wtilde{V}^* h=(\ol{Q}V)^*h=V^*(\ol{Q}^*h)=T^*Q^*h=(QT)^*h
 \end{align*}
 for all  $h\in\mcl{H}$, i.e., $\wtilde{V}^*|_{\mcl{H}}=(QT)^*$. Thus $\wtilde{V}$ is an isometric lifting of $QT$. Set $\what V=\sMatrix{\wtilde{V}&0\\0&V}\in\B{\wtilde{\mcl{K}}\oplus\mcl{K}}$. Clearly $\what{V}$ is an isometric lifting of the contraction $\what{T}$. Since $\what{T}\what{X}=\what{X}\what{T}$, by commutant lifting theorem there exists $\what{Y}\in\B{\wtilde{\mcl{K}}\oplus\mcl{K}}$  such that $\what V\what Y=\what Y\what V$, $\what{Y}^*|_{\mcl{H}\oplus\mcl{H}}=\what{X}^*$  and $\snorm{\what Y}=\snorm{ \what X}$.  Let $\what{Y}=\sMatrix{\ast&B\\\ast&\ast}\in\B{\wtilde{\mcl{K}}\oplus\mcl{K}}$ where $B=\sMatrix{Y &Y_1&Y_2&Y_3\dots}^{tr}\in\B{\mcl{K},\wtilde{\mcl{K}}}$ with respect to the decomposition $\wtilde{\mcl{K}}=\mcl{K}\bigoplus(\bigoplus_1^\infty \mcl{K}_0)$. Then
 \begin{align*}
     \what V\what Y=\what Y\what V
     ~~\Longrightarrow~~
     \wtilde{V}B=BV
     ~~\Longrightarrow~~
     \ol{Q}VY=YV,
 \end{align*}
 where $Y\in\B{\mcl K}$. Also 
 \begin{align*}
     \what{Y}^*|_{\mcl{H}\oplus\mcl{H}}=\what{X}^*
     ~~\Longrightarrow~~
     B^*|_{\mcl{H}}=X^*
     ~~\Longrightarrow~~
     Y^*|_{\mcl{H}}=X^*,
 \end{align*}
 so that $Y$ is a lifting of $X$. Hence $\norm{X}\leq\norm{Y}\leq\norm{B}\leq\snorm{\what{Y}}=\snorm{\what{X}}=\norm{X}$. \\
 %%%%%%%%% CASE (2) %%%%%%%%%%%%%
 (ii)  Set  $\what T=\sMatrix{TQ&0\\0&T}$ on $\mcl{H}\oplus\mcl{H}$. Let $\what{X}, D, \wtilde{\mcl{K}}$ be as in case (i) with $\mcl{K}_0=\clran{VD}\subseteq\mcl{K}$. Define $\wtilde{V}\in\B{\wtilde{\mcl{K}}}$ by
 \begin{align*}  
         \wtilde{V}=\Matrix{
       			 V\ol{Q}  &  0                   &             0 & \dots   & \cdots\\
       			 VD  &  0                   &              0 & \dots  & \cdots \\
       		         0     & I_{\mcl{K}_0} &              0 & \dots  & \cdots \\
       		         0     &                  0  & I_{\mcl {K}_0} & \dots &  \cdots\\
       			 \vdots& \vdots        & \vdots       &  \ddots    & \vdots         
       		    }.
 \end{align*} 
 Note that $\what V=\sMatrix{\wtilde{V}&0\\0&V}\in\B{\wtilde{\mcl{K}}\oplus\mcl{K}}$ is an isometric lifting of the contraction $\what{T}$, and since $\what{T}\what{X}=\what{X}\what{T}$, by proceeding as in case (i) we can get  $Y\in\B{\mcl{K}}$ such that $V\ol{Q}Y=YV, Y^*|_{\mcl{H}}=X^*$ and $\norm{Y}=\norm{X}$. \\
 %%%%%%%%% CASE (3) %%%%%%%%%%%
 (iii) Suppose $Q$ is a unitary. Set $\what T=\sMatrix{TQ^*&0\\0&T}$ and $\what X=\sMatrix{0&0\\X&0}$ on $\mcl{H}\oplus\mcl{H}$, and $\what{V}=\sMatrix{V\ol{Q}^*&0\\0&V}$ on $\mcl{K}\oplus\mcl{K}$. Note that $(\what{V},\mcl{K}\oplus\mcl{K})$ is an isometric lifting of $\what{T}$. Since $\what{T}\what{X}=\what{X}\what{T}$, by commutant lifting theorem there exists a lifting  $\what{Y}=\sMatrix{\ast&\ast\\ Y&\ast}\in\B{\mcl{K}\oplus\mcl{K}}$ of $\what{X}$ such that $\what{Y}\what{V}=\what{V}\what{Y}$ and $\snorm{\what{Y}}=\snorm{\what{X}}$. Observe that $Y$ is the required lifting of $X$. This completes the proof. 
\end{proof}

\begin{rmk}
 In Theorem \ref{thm-CLT}(iii) suppose $Q$ is only a co-isometry, so that $TX=XTQ^*$. The above proof shows that, in such case  also we can get a lifting $Y$ of $X$ satisfying all properties except the equality $VY\ol{Q}=YV$, but  we get $VY=YV\ol{Q}^*$. 
\end{rmk}

\begin{thm}[Q-intertwining lifting]\label{thm-CIL}
 Let $T_i\in\B{\mcl{H}_i}$ be a contraction with isometric lifting $(V_i,\mcl{K}_i),$ $i=1,2$, and let $X\in\B{\mcl{H}_1,\mcl{H}_2}$. Suppose $Q\in\B{\mcl{H}_2}$ and $\ol{Q}\in\B{\mcl{K}_2}$ are contractions. 
 \begin{enumerate}
   \item [(i)]
         If  $XT_1=QT_2X$, then there exists a lifting $Y\in\B{\mcl{K}_1,\mcl{K}_2}$ of $X$ such that  $YV_1=\ol{Q}V_2Y$.
   \item [(ii)]
         If  $XT_1=T_2QX$, then there exists a lifting $Y\in\B{\mcl{K}_1,\mcl{K}_2}$ of $X$ such that $YV_1=V_2\ol{Q}Y$.
 \end{enumerate}
 Suppose $Q\in\B{\mcl{H}_1}$ and $\ol{Q}\in\B{\mcl{K}_1}$ are unitary. 
  \begin{enumerate}
 \item [(iii)]
        If $XT_1=T_2XQ$, then there exists a lifting $Y\in\B{\mcl{K}_1,\mcl{K}_2}$ of $X$ such that $YV_1=V_2Y\ol{Q}$.
 \end{enumerate}
  In all cases $T_2^nX=P_{\mcl{H}_2}V_2^nY|_{\mcl{H}_1}$ and $XT_1^n=P_{\mcl{H}_2}YV_1^n|_{\mcl{H}_1}$ for all $n\geq 0$. Moreover,  $Y$ can be chosen such that $\norm{Y}=\norm{X}$. 
 
\end{thm}

\begin{proof} 
 First assume that $XT_1=QT_2X$.  Set 
 \begin{align*}
   & \what{T}=\sMatrix{T_1&0\\0&T_2}, 
      ~\what{X}=\sMatrix{0&0\\ X&0},
      ~\mcl{Q}=\sMatrix{I_{\mcl{H}_1}&0\\0& Q}~~\in\B{\mcl{H}_1\oplus \mcl{H}_2},
       \quad\mbox{ and}\\
       &\ol{\mcl{Q}}=\sMatrix{I_{\mcl{K}_1}&0\\0&\ol{Q}},
       ~\what {V}=\sMatrix{V_1&0\\0&V_2}~~\in\B{\mcl{K}_1\oplus\mcl{K}_2}.
 \end{align*}
 Note that $\mcl{H}_1\oplus\mcl{H}_2$ is reducing for the contraction $\ol{\mcl{Q}}$, and $\ol{\mcl{Q}}|_{\mcl{H}_1\oplus\mcl{H}_2}=\mcl{Q}$.  Since $\what{X}\what{T}=\mcl{Q}\what{T}\what{X}$ and  $\what{V}$ is an isometric lifting of $\what{T}$, by Theorem \ref{thm-CLT} there exists a lifting $\what{Y}=\sMatrix{\ast& \ast\\ Y& \ast}\in\B{\mcl K_1\oplus\mcl K_2}$ of $\what{X}$ such that  $\what{Y}\what{V}=\ol{\mcl{Q}}\what{V}\what{Y}$ and $\snorm{\what{Y}}=\snorm{\what{X}}$.  As $\what{Y}\what{V}=\ol{\mcl{Q}}\what{V}\what{Y}$ we get $YV_1=\ol{Q}V_2Y$. Also since $\what{Y}^*|_{\mcl{H}_1\oplus\mcl{H}_2}=\what{X}^*$ we have $Y^*|_{\mcl{H}_2}=X^*$, i.e., $Y$ is a lifting of $X$. Hence  $\norm{X}\leq\norm{Y}\leq\snorm{\what{Y}}=\snorm{\what{X}}=\norm{X}$. The case when $T_2QX=XT_1$ can be proved similarly since $\what{T}\mcl{Q}\what{X}=\what{X}\what{T}$. For the case when $XT_1=T_2XQ$  repeat the above process by taking $\mcl{Q}=\sMatrix{Q&0\\0& I_{\mcl{H}_2}}\in\B{\mcl{H}_1\oplus\mcl{H}_2}$ and $\ol{\mcl{Q}}=\sMatrix{\ol{Q}&0\\0&I_{\mcl{K}_2}}\in\B{\mcl{K}_1\oplus\mcl{K}_2}$. 
\end{proof}

\begin{rmk}
  Theorem \ref{thm-CIL} can also be deduced from the classical intertwining theorem as follows: To prove $(i)$ of Theorem \ref{thm-CIL} let $(\wtilde{V},\wtilde{\mcl{K}})$ be an isometric lifting of the contraction $\ol{Q}V_2=\sMatrix{QT_2&0\\ \ast &\ast}\in\B{\mcl{K}_2}$. Then $(\wtilde{V},\wtilde{\mcl{K}})$ is also an isometric lifting of $QT_2$. Since $XT_1 = (QT_2)X$, the classical intertwining lifting theorem yields a lifting $\what{Y}\in\B{\mcl{K}_1,\wtilde{\mcl{K}}}$ of $X$ such that $\what{Y}V_1=\wtilde{V}\what{Y}$ and $\snorm{\what{Y}}=\norm{X}$. With respect to the decomposition $\wtilde{\mcl{K}}=\mcl{K}_2\oplus\mcl{K}_2^\perp$ let $\wtilde{V}=\sMatrix{\ol{Q}V_2&0\\\ast&\ast}$ and $\what{Y}=\sMatrix{Y\\ Y_0}$ with $Y\in\B{\mcl{K}_1,\mcl{K}_2}$.  Then $\what{Y}V_1=\wtilde{V}\what{Y}$  implies that  $YV_1=\ol{Q}V_2Y$. Also, since $\what{Y}$ is a lifting of $X$ we have $Y$ is a lifting of $X$, so that $\norm{X}\leq\norm{Y}\leq\snorm{\what{Y}}=\norm{X}$. Part  $(ii)$ is proved similarly by replacing $\ol{Q}V_2$ with $V_2\ol{Q}$. For $(iii)$, rewrite the assumption as $X(T_1Q^*)=T_2X$, and note that $V_1\ol{Q}^*$ is already an isometric lifting of $T_1Q^*$, so the result follows even more directly from classical intertwining lifting theorem. Specializing Theorem \ref{thm-CIL} to $T_1=T_2$ and $V_1=V_2$, we obtain Theorem \ref{thm-CLT}.  
\end{rmk}

 Suppose $T\in\B{\mcl{H}}$. Recall that $(V,\mcl{K})$ is an isometric lifting of a $T$ if and only if $(V^*,\mcl{K})$ is an co-isometric extension of $T^*$. So we can restate the Theorems \ref{thm-CLT}, \ref{thm-CIL} as follows. We will be using this versions later. 

\begin{thm}[$Q$-commutant extension]\label{thm-CET}
 Let $T\in\B{\mcl{H}}$ be a contraction with co-isometric extension $(W,\mcl{K})$, and let $X\in\B{\mcl{H}}$. Suppose $Q\in\B{\mcl{H}}$ and $\ol{Q}\in\B{\mcl{K}}$ are contractions.   
  \begin{enumerate}
   \item [(i)]
         If $XTQ=TX$, then there exists an extension $Y\in\B{\mcl{K}}$ of $X$ such that $YW\ol{Q}=WY$.
   \item [(ii)]
         If $XQT=TX$, then there exists an extension $Y\in\B{\mcl{K}}$ of $X$ such that $Y\ol{Q}W=WY$. 
 \end{enumerate}    
 Further assume that $Q$ and $\ol{Q}$ are unitary. 
 \begin{enumerate}
  \item [(iii)]
         If $QXT=TX$, then there exists an extension $Y\in\B{\mcl{K}}$ of $X$ such that $\ol{Q}YW=WY$.    
 \end{enumerate}
 In all cases $T^nX^m=P_\mcl{H}W^nY^m|_\mcl{H}$ and $X^nT^m=P_\mcl{H}Y^nW^m|_\mcl{H}$ for all $n,m\geq 0$. Moreover,  $Y$ can be chosen such that $\norm{Y}=\norm{X}$.  
\end{thm}

\begin{thm}[Q-intertwining extension]
 Let $T_i\in\B{\mcl{H}_i}$ be a contraction with co-isometric extension $(W_i,\mcl{K}_i),~i=1,2$ and let $X\in\B{\mcl{H}_1,\mcl{H}_2}$. Suppose $Q\in\B{\mcl{H}_1}$ and $\ol{Q}\in\B{\mcl{K}_1}$ are contractions.  
\begin{enumerate}
   \item [(i)]
         If $XT_1Q=T_2X$, then there exists an extension  $Y\in\B{\mcl{K}_1,\mcl{K}_2}$ of $X$ such that $YW_1\ol{Q}=W_2Y$.
   \item [(ii)]
	  If  $XQT_1=T_2X$, then there exists an extension  $Y\in\B{\mcl{K}_1,\mcl{K}_2}$ of $X$ such that $Y\ol{Q}W_1=W_2Y$.
 \end{enumerate}
 Suppose $Q\in\B{\mcl{H}_2}$ and $\ol{Q}\in\B{\mcl{K}_2}$ are unitary.
  \begin{enumerate}
   \item [(iii)]
         If $QXT_1=T_2X$, then there exists an extension  $Y\in\B{\mcl{K}_1,\mcl{K}_2}$ of $X$ such that $\ol{Q}YW_1=W_2Y$.
 \end{enumerate}
 In all cases $T_2^nX=P_{\mcl{H}_2}V_2^nY|_{\mcl{H}_1}$ and $XT_1^n=P_{\mcl{H}_2}YV_1^n|_{\mcl{H}_1}$ for all $n\geq 0$. Moreover,  $Y$ can be chosen such that $\norm{Y}=\norm{X}$. 
\end{thm}

\begin{rmk}
 Note that case $(i)$ is a stronger version of \cite[Theorem 3]{Seb94b}. In \cite{Seb94b} Sebestyen considered $\ol{Q}\in\B{\mcl{K}_1}$ with the additional assumption that $\cspan\{W_1^{*k}h: h\in\mcl{H}_1, 0\leq k\leq n\}$ reduces $\ol{Q}$ for every $n\geq 0$.  
\end{rmk} 
 
 Recall that the minimal isometric dilation is an isometric lifting.  Thus, Theorem \ref{thm-CLT} characterizes the operators $X$ which are  $Q$-commutant to $T$ in terms of the operators $Y$ which are $\ol{Q}$-commutant to the minimal isometric dilation $V$ of $T$.  Next we characterizes the operators $X$ which are  $Q$-commutant to $T$ in terms of the operators $Y$ which are $\ol{Q}$-commutant to the minimal unitary dilation of $T$, provided  $Q$ is a unitary. To prove our result we use the following lemma.

\begin{lem}[\cite{DMP68}]\label{lem-UD-1}
 Suppose $T\in\B{\mcl{H}}$ is a contraction with the unique minimal co-isometric extension $(W,\mcl{K}_0)$. Let $(U^*,\mcl{K})$ be the  unique minimal co-isometric extension of $W^*$. Then $U^*$ is a unitary, and $(U,\mcl{K})$ is the unique minimal unitary dilation of $T$. 
\end{lem}

\begin{thm}\label{thm-unitary-CDT}
 Let $T\in \B{\mcl H}$ be a contraction with the  minimal unitary dilation $(U,\mcl{K})$. Suppose $Q\in\B{\mcl{H}}$ is a unitary and let $X\in \B{\mcl{H}}$. 
 \begin{enumerate}[label=(\roman*)]
   \item 
          If $XT=QTX$, then there exist $\ol{Q}\in\B{\mcl{K}}$ unitary and  a dilation $Y\in\B{\mcl{K}}$ of $X$ such that $YU=\ol{Q}UY$.
   \item 
         If $XT=TXQ$, then there exist $\ol{Q}\in\B{\mcl{K}}$ unitary and a dilation $Y\in\B{\mcl{K}}$ of $X$ such that $YU=UY\ol{Q}$.         
 \end{enumerate}
 In fact, given $q\in\mbb{T}$ we can choose $\ol{Q}=Q\oplus qI_{\mcl{H}^\perp}$. In all cases $X^nT^m=P_\mcl{H}Y^nU^m|_\mcl{H}$ and $T^nX^m=P_\mcl{H}U^nY^m|_\mcl{H}$ for all $n,m\geq 0$. Moreover,  $Y$ can be chosen such that $\norm{Y}=\norm{X}$. 
\end{thm}

\begin{proof}
 We prove only the case $(i)$. Case $(ii)$ can be proved similarly. Suppose $q\in\mbb{T}$ and $(W,\mcl{K}_0)$ is the minimal co-isometric extension  of $T$. From Lemma \ref{lem-UD-1} and the uniqueness of minimal unitary dilation we can assume that $(U^*,\mcl{K})$ is the minimal co-isometric extension of $W^*$. Note that $\mcl{H}\subseteq\mcl{K}_0\subseteq\mcl{K}$.  Let $Q_0=Q\oplus qI_{\mcl{K}_0\ominus\mcl{H}}\in\B{\mcl{K}_0}$. Since $Q^*XT=TX$, by Theorem \ref{thm-CET} (iii) there exists an extension $Y_0\in\B{\mcl{K}_0}$ of $X$ with $\norm{Y_0}=\norm{X}$ such that $Q_0^*Y_0W=WY_0$, $T^nX^m=P_\mcl{H}W^nY_0^m|_\mcl{H}$ and $X^nT^m=P_\mcl{H}Y_0^nW^m|_\mcl{H}$ for all $n,m\geq 0$. Again since $Y_0^*W^*Q_0^*=W^*Y_0^*$, by Theorem \ref{thm-CET} (i) there exists an extension $Y^*\in\B{\mcl{K}}$ of $Y_0^*$ with $\norm{Y^*}=\norm{Y_0^*}$ such that $Y^*U^*(Q\oplus qI_{\mcl{K}\ominus\mcl{H}})^*=U^*Y^*$, $W^{\ast n}Y_0^{\ast m}=P_{\mcl{K}_0}U^{\ast n}Y^{\ast m}|_{\mcl{K}_0}$ and $Y_0^{\ast n}W^{\ast m}=P_{\mcl{K}_0}Y^{\ast n}U^{\ast m}|_{\mcl{K}_0}$ for all $n,m\geq 0$.  Note that $Y$ has the required properties.  
\end{proof}

\begin{thm}\label{thm-unitary-CIT}
 Let $T_i\in \B{\mcl H_i}$ be a contraction with the  minimal unitary dilation $(U_i,\mcl{K}_i),~i=1,2$ and $X\in \B{\mcl{H}_1,\mcl{H}_2}$. 
 \begin{enumerate}[label=(\roman*)]
   \item
        Suppose $Q\in\B{\mcl{H}_2}$ is a unitary such that  $XT_1=QT_2X$. Then there exist $\ol{Q}\in\B{\mcl{K}_2}$ unitary and $Y\in\B{\mcl{K}_1,\mcl{K}_2}$ such that $YU_1=\ol{Q}U_2Y$.
   \item
        Suppose $Q\in\B{\mcl{H}_1}$ is a unitary such that  $XT_1=T_2XQ$. Then there exist $\ol{Q}\in\B{\mcl{K}}$ unitary and $Y\in\B{\mcl{K}_1,\mcl{K}_2}$ such that $YU_1=U_2Y\ol{Q}$.     
 \end{enumerate}
 In fact, given $q\in\mbb{T}$ we can choose $\ol{Q}=Q\oplus qI$.  In all cases $XT_1^n=P_{\mcl{H}_2}YU_1^n|_{\mcl{H}_1}$ and $T_2^nX=P_{\mcl{H}_2}U_2^nY|_{\mcl{H}_1}$ for all $n\geq 0$. Moreover,  $Y$ can be chosen such that $\norm{Y}=\norm{X}$. 
\end{thm}

 Proof is similar to that of Theorem \ref{thm-CIL}. % Note that it cannot be done as in remark 2.8 after theorem thm-CIL. 

\subsection{Dilation theorems}

 In this section we prove an analogue  of Ando's dilation theorem for $Q$-commuting contractions. 
    
\begin{thm}[$Q$-commuting isometric dilation]\label{thm-Q-CID}
 Let $T_1,T_2\in\B{\mcl{H}}$ be contractions and $Q\in\B{\mcl{H}}$ be a unitary such that  $T_2T_1=QT_1T_2$. Then there exists a Hilbert space $\mcl{K}\supseteq\mcl{H}$, isometries $V_1,V_2\in\B{\mcl{K}}$ and $\ol{Q}\in\B{\mcl{K}}$ unitary  such that 
     \begin{enumerate}[label=(\roman*)]
     \item $V_2V_1=\ol{Q}V_1V_2$; and 
     \item $V_i$ is a lifting (and hence a dilation) of $T_i$ so that $T_1^nT_2^m=P_\mcl{H}V_1^nV_2^m|_{\mcl{H}}$ and $T_2^nT_1^m=P_\mcl{H}V_2^nV_1^m|_{\mcl{H}}$ for all $n,m\geq 0$. 
     \end{enumerate}
  In fact, given $q\in\mbb{T}$ we can choose $\ol{Q}=Q\oplus qI_{\mcl{K}\ominus\mcl{H}}$.    
\end{thm}

\begin{proof} 
 Fix $q\in\mbb{T}$. Let $(\what{V}_1,\what{\mcl{K}})$ be the minimal isometric dilation of $T_1$. Since $T_2T_1=QT_1T_2$, by Theorem \ref{thm-CLT}(i) there exists $\what{V}_2\in\B{\what{\mcl{K}}}$ such that 
 \begin{align*}
     \what{V}_2\what{V}_1=(Q\oplus qI_{\what{\mcl{K}}\ominus\mcl{H}})\what{V}_1\what{V}_2,
     \quad\what{V}_2^*|_{\mcl{H}}=T_2^*
     \quad\mbox{and }
     \snorm{\what{V}_2}=\norm{T_2}\leq 1.
 \end{align*}
 Suppose $(V_2,\mcl{K})$ is the minimal isometric dilation of $\what{V}_2$. Note that $\mcl{H}\subseteq\what{\mcl{K}}\subseteq\mcl{K}$. Since $\what{V}_1\what{V}_2=(Q\oplus qI_{\what{\mcl{K}}\ominus\mcl{H}})^*\what{V}_2\what{V}_1$, from Theorem \ref{thm-CLT}(i) we get $V_1\in\B{\mcl{K}}$ such that 
 \begin{align*}
    V_1V_2=(Q\oplus qI_{\mcl{K}\ominus\mcl{H}})^*V_2V_1,
    \quad V_1^*|_{\what{{\mcl{K}}}}=\what{V}_1^*
    \quad\mbox{and } \norm{V_1}=\snorm{\what{V}_1}\leq 1.
 \end{align*}
 Let $V_1=\sMatrix{\what{V}_1&0\\ A& B}$ w.r.t the decomposition $\mcl{K}=\what{\mcl{K}}\oplus\what{\mcl{K}}^\perp$. Since
 \begin{align*}
   0\leq \what{V}_1^*\what{V}_1+A^*A\leq\snorm{\what{V}_1^*\what{V}_1+A^*A}I\leq\norm{V_1^*V_1}I\leq I 
 \end{align*}
 and $\what{V}_1$ is an isometry we have $A=0$, so that  $V_1|_{\what{\mcl{K}}}=\what{V}_1$. Since $Q$ and $V_2$ are isometries we have  
 \begin{align*}
   \norm{V_1V_2^nk} = \norm{(Q\oplus qI_{\mcl{H}^\perp})V_1V_2^nk}
                     = \norm{V_2V_1V_2^{n-1}k}
                     = \norm{V_1V_2^{n-1}k}\quad\forall~n\geq 1. 
  \end{align*}                   
 Repeating the above step recursively we get 
 \begin{align*}  
     \norm{V_1V_2^nk}
                     = \norm{V_1V_2k}
                     = \norm{(Q\oplus qI_{\mcl{H}^\perp})V_1V_2k}
                     = \norm{V_2V_1k}=\norm{V_1k}=\snorm{\what{V}_1k}=\norm{k} 
                    = \norm{V_2^nk}
 \end{align*}
 for every $k\in\what{\mcl{K}}$ and $n\geq 1$. Clearly the above equality holds for $n=0$. Hence 
 \begin{align*}
   \snorm{(I-V_1^*V_1)^{\frac{1}{2}}V_2^nk}^2
           %=\bip{V_2^nk,(I-V_1^*V_1)V_2^nk} \\
          =\ip{V_2^nk,V_2^nk}-\ip{V_2^nk,V_1^*V_1V_2^nk}
          =\norm{V_2^nk}^2-\norm{V_1V_2^nk}^2
          =0
 \end{align*}
 for all $k\in\what{\mcl{K}}$ and $n\geq 0$. Since $\mcl{K}=\cspan\{V_2^nk: k\in\what{\mcl{K}}, n\geq 0\}$ from above equation we get $(I-V_1^*V_1)=0$, i.e., $V_1$ is an isometry. Moreover, since minimal isometric dilations are liftings we have $V_i^*|_{\mcl{H}}=(V_i^*|_{\what{\mcl{K}}})|_{\mcl{H}}=\what{V}_i^*|_{\mcl{H}}=T_i^*$ for $i=1,2$. 
 %\begin{align*}
    %V_1^*|_{\mcl{H}}=(V_1^*|_{\what{\mcl{K}}})|_{\mcl{H}}=\what{V}_1^*|_{\mcl{H}}=T_1^*\\
    %V_2^*|_{\mcl{H}}=(V_2^*|_{\what{\mcl{K}}})|_{\mcl{H}}=\what{V}_2^*|_{\mcl{H}}=T_2^*.
 %\end{align*}
 Thus $V_1, V_2$ are isometric lifting of $T_1, T_2$ respectively, so that $(ii)$ follows. This completes the proof.  
\end{proof}

\begin{cor}[$Q$-commuting co-isometric dilation]\label{cor-QCD}
 Let $T_1,T_2\in\B{\mcl{H}}$ be contractions and $Q\in\B{\mcl{H}}$ be a unitary such that $T_2T_1=T_1T_2Q$. Then there exists a Hilbert space $\mcl{K}\supseteq\mcl{H}$, co-isometries $W_1,W_2\in\B{\mcl{K}}$ and $\ol{Q}\in\B{\mcl{K}}$ unitary such that 
      \begin{enumerate}[label=(\roman*)]
           \item $W_2W_1=W_1W_2\ol{Q}$; and
          \item $W_i$ is an extension (and hence a dilation) of $T_i$ so that $T_1^nT_2^m=P_\mcl{H}W_1^nW_2^m|_{\mcl{H}}$ and $T_2^nT_1^m=P_\mcl{H}W_2^nW_1^m|_{\mcl{H}}$ for all $n,m\geq 0$. 
     \end{enumerate}  
      In fact, given $q\in\mbb{T}$ we can choose $\ol{Q}=Q\oplus qI_{\mcl{K}\ominus\mcl{H}}$.
\end{cor} 

\begin{proof}  
 Fix $q\in\mbb{T}$.  Since $T_1^*T_2^*=Q^*T_2^*T_1^*$, by Theorem \ref{thm-Q-CID}  there exists Hilbert space $\mcl{K}\supseteq\mcl{H}$ and an isometric lifting  $W_i^*\in\B{\mcl{K}}$ of $T_i^*$ such that $W_1^*W_2^*=(Q\oplus qI_{\mcl{H}^\perp})^*W_2^*W_1^*$. Note that $W_i\in\B{\mcl{K}},~i=1,2$ are the required co-isometric extensions. 
\end{proof}

\begin{thm}\label{thm-Q-iso-uni} 
 Let $V_1,V_2\in\B{\mcl{H}}$ be isometries and $Q\in\B{\mcl{H}}$ be a unitary such that  $V_2V_1=QV_1V_2$. Then there exists a Hilbert space $\mcl{K}\supseteq\mcl{H}$ and unitaries $\ol{Q},U_1,U_2\in\B{\mcl{K}}$ such that 
  \begin{enumerate}[label=(\roman*)]
              \item $U_2U_1=\ol{Q}U_1U_2$; and
              \item $U_i$ is an extension (and hence a dilation) of $V_i$ so that $V_1^nV_2^m=P_\mcl{H}U_1^nU_2^m|_{\mcl{H}}$ and $V_2^nV_1^m=P_\mcl{H}U_2^nU_1^m|_{\mcl{H}}$ for all $n,m\geq 0$. 
  \end{enumerate}
  In fact, given $q\in\mbb{T}$ we can choose $\ol{Q}=Q\oplus qI_{\mcl{K}\ominus\mcl{H}}$.
\end{thm}

\begin{proof} 
 Fix $q\in\mbb{T}$. Suppose $(\what{V}_2,\what{\mcl{H}})$ is the minimal unitary dilation of $V_2$.  Then  $\what{\mcl{H}}=\cspan\{\what{V}_2^n\mcl{H}: n\in\mbb{Z}\}$ and $\what{V}_2$ is an extension of $V_2$. Define $\what{V}_1:\what{\mcl{H}}\to\what{\mcl{H}}$ by 
 \begin{align*}
     \what{V}_1(\what{V}_2^nh)=\big(\what{Q}^*\what{V}_2\big)^nV_1h\qquad\forall~h\in\mcl{H}, n\in\mbb{Z},
 \end{align*}
 where $\what{Q}=(Q\oplus qI_{\what{\mcl{H}}\ominus\mcl{H}})\in\B{\what{\mcl{H}}}$. Then for all $h,h'\in\mcl{H}$ and $n\geq m$,
 \begin{align*}
   \bip{\what{V}_1(\what{V}_2^nh),\what{V}_1(\what{V}_2^mh')}
           %&=\bip{\big(\what{Q}^*\what{V}_2\big)^nV_1h,\big(\what{Q}^*\what{V}_2\big)^mV_1h'}  \\
           &=\bip{\big(\what{Q}^*\what{V}_2\big)^{n-m}V_1h,V_1h'} \\
           &=\bip{\big(\what{Q}^*\what{V}_2\big)^{n-m-1}\what{Q}^*\what{V}_2V_1h,V_1h'}\\
           &=\bip{\big(\what{Q}^*\what{V}_2\big)^{n-m-1}Q^*V_2V_1h,V_1h'} \\
           &=\bip{\big(\what{Q}^*\what{V}_2\big)^{n-m-1}V_1V_2h,V_1h'}\\
           &=\bip{V_1V_2^{n-m}h,V_1h'}\quad\quad\qquad(\mbox{ by repeating above steps})\\
           %&=\bip{V_2^{n-m}h,h'}\quad\qquad\quad\qquad(\because V_1\textnormal{ is  isometry}) \\
           &=\bip{\what{V}_2^{n-m}h,h'}\qquad\qquad\qquad(\because \what{V}_2|_\mcl{H}=V_2 \mbox{ and }n-m\geq 0) \\
           &=\ip{\what{V}_2^nh,\what{V}_2^mh'}.
 \end{align*}
 Thus $\what{V}_1$ is a well defined isometry. Clearly $\what{V}_1$ is an extension of $V_1$. Moreover, $\what{Q}\what{V}_1\what{V}_2=\what{V}_2\what{V}_1$ on $\what{\mcl{H}}$. Suppose $(U_1, \mcl{K})$ is the minimal unitary dilation (and hence an extension) of $\what{V}_1$, so that $\mcl{K}=\cspan\{U_1^n(\what{\mcl{H}}): n\in\mbb{Z}\}=\cspan\{U_1^n(\what{\mcl{H}}): n\leq 0\}$ as $U_1$ leaves $\what{\mcl{H}}$ invariant. Define $U_2:\mcl{K}\to\mcl{K}$ by 
  \begin{align*}
     U_2(U_1^n\what{h})=(\ol{Q}U_1)^n\what{V}_2\what{h}\qquad\forall~\what{h}\in\what{\mcl{H}},n\in\mbb{Z},
 \end{align*}
 where $\ol{Q}=(Q\oplus qI_{\mcl{K}\ominus\mcl{H}})\in\B{\mcl{K}}$. As in the case of $\what{V}_1$, it can be verified that $U_2$ is also a well defined isometric extension of $\what{V}_2$. Clearly $U_2U_1=\ol{Q}U_1U_2$.  Now we shall prove that $U_2$ is onto, so that it is a  unitary. For, if $n> 0$ let  
 \begin{align*}
   \mcl{K}_n=\cspan\{U_1^{*j}\what{h}:0\leq j\leq n, \what{h}\in\what{\mcl{H}}\}
            =\lspan\{U_1^{*j}\what{h}:0\leq j\leq n, \what{h}\in\what{\mcl{H}}\}.
 \end{align*}
 We prove by induction that $U_2$ maps $\mcl{K}_n$ onto $\mcl{K}_n$ for every $n> 0$. Suppose $n=1$. Then for all $0\leq j\leq 1$ and $\what{h}\in\what{\mcl{H}}$ we have $U_1^{*j}\what{V}_2^\ast \what{Q}^j\what{h}\in\mcl{K}_1$ and    
 \begin{align*}
    U_2(U_1^{*j}\what{V}_2^\ast \what{Q}^j\what{h})
       =(\ol{Q}U_1)^{*j}\what{V}_2\what{V}_2^*\what{Q}^j\what{h}
       =U_1^{*j}\ol{Q}^{*j}\what{Q}^j\what{h}
       =U_1^{*j}\what{h}.
 \end{align*}
 Thus $U_2(\mcl{K}_1)=\mcl{K}_1$. Now assume that $U_2$ maps $\mcl{K}_n$ onto $\mcl{K}_n$. To prove that $U_2$ maps $\mcl{K}_{n+1}$ onto $\mcl{K}_{n+1}$ it is enough to prove that $U_1^{\ast (n+1)}\what{h}$ has a pre-image for every $\what{h}\in\what{\mcl{H}}$. Since $U_1^{*n}\what{h}\in\mcl{K}_n$ there exists $x\in\mcl{K}_n$ such that $U_2(x)=U_1^{\ast n}\what{h}$. Note that $\mcl{H}\subseteq\what{\mcl{H}}\subseteq\mcl{K}_n\subseteq\mcl{K}$, hence $\mcl{K}_n$ is reducing for  $\ol{Q}$. Therefore there exists  $z\in\mcl{K}_n$ such that $U_2(z)=\ol{Q}U_2x$. Clearly $U_1^*(z)\in\mcl{K}_{n+1}$, and 
 \begin{align*}
    U_2U_1^*z = U_1^*\ol{Q}^*U_2(z)
                       = U_1^*U_2x
                       = U_1^{\ast n+1}\what{h}.
 \end{align*}
 Thus $U_2$ maps $\mcl{K}_{n+1}$ onto $\mcl{K}_{n+1}$. By induction we conclude that $U_1^{*n}\what{h}$ has a pre-image under $U_2$ for all $n> 0,\what{h}\in\mcl{\what{H}}$. Since $\mcl{K}=\cspan\{U_1^n(\what{\mcl{H}}): n\leq 0\}$ we conclude that $U_2$ is onto. Note that $(\ol{Q},U_1,U_2,\mcl{K})$ is the required quadruple. 
\end{proof}

\begin{cor}\label{cor-Q-coiso-uni} 
 Let $W_1,W_2\in\B{\mcl{H}}$ be co-isometries and $Q\in\B{\mcl{H}}$ be a unitary such that  $W_2W_1=W_1W_2Q$. Then there exist a Hilbert space $\mcl{K}\supseteq\mcl{H}$ and unitaries $\ol{Q},U_1,U_2\in\B{\mcl{K}}$ such that 
  \begin{enumerate}[label=(\roman*)]
              \item $U_2U_1=U_1U_2\ol{Q}$; and
              \item $U_i$ is a lifting (and hence a dilation) of $W_i$ so that $W_1^nW_2^m=P_\mcl{H}U_1^nU_2^m|_{\mcl{H}}$ and $W_2^nW_1^m=P_\mcl{H}U_2^nU_1^m|_{\mcl{H}}$ for all $n,m\geq 0$. 
  \end{enumerate}
  In fact, given $q\in\mbb{T}$ we can choose $\ol{Q}=Q\oplus qI_{\mcl{K}\ominus\mcl{H}}$.
\end{cor}
  
\begin{proof}
 Since $W_i^*$'s are isometries satisfying $W_2^*W_1^*=QW_1^*W_2^*$, from  above theorem there exists Hilbert space $\mcl{K}\supseteq\mcl{H}$ and unitaries $U_1,U_2\in\B{\mcl{K}}$ such that $U_i^*$'s are extensions of $W_i^*$'s with $U_2^*U_1^*=(Q\oplus qI_{\mcl{H}^\perp})U_1^*U_2^*$. 
\end{proof} 
 
 Combining Theorems \ref{thm-Q-CID}, \ref{thm-Q-iso-uni} and Corollaries \ref{cor-QCD}, \ref{cor-Q-coiso-uni} we have the following analogue of Ando's theorem for $Q$-commuting contractions. 
  
\begin{thm}[$Q$-commuting unitary dilation]\label{thm-main}
 Let $T_1,T_2\in\B{\mcl{H}}$ be contractions and $Q\in\B{\mcl{H}}$ be a unitary such that $T_2T_1=QT_1T_2$ (respectively $T_2T_1=T_1T_2Q$). Then there exist Hilbert space $\mcl{K}\supseteq\mcl{H}$ and unitaries $\ol{Q},U_1,U_2\in\B{\mcl{K}}$ such that 
 \begin{enumerate}[label=(\roman*)]
     \item $U_2U_1=\ol{Q}U_1U_2$ (respectively $U_2U_1=U_1U_2\ol{Q}$); and
     \item $T_1^nT_2^m=P_\mcl{H}U_1^nU_2^m|_{\mcl{H}}$ and $T_2^nT_1^m=P_\mcl{H}U_2^nU_1^m|_{\mcl{H}}$ for all $n,m\geq 0$. 
 \end{enumerate}
 In fact, given $q\in\mbb{T}$ we can choose $\ol{Q}=Q\oplus qI_{\mcl{K}\ominus\mcl{H}}$.
\end{thm}

\begin{rmk}
 If $Q=qI_{\mcl{H}}$ for some $q\in\mbb{T}$, then Theorem \ref{thm-main} reduces to $q$-commuting dilation theorem. 
\end{rmk}

\begin{ackn}
  The first author thanks the Department of Atomic Energy (DAE), Government of India for financial support and IMSc Chennai for providing necessary facilities to carry out this work. We thank B. V. R. Bhat for helpful suggestions.
\end{ackn}

%
%\noindent
%\textsc{Nirupama Mallick}:
%{\small\itshape The Institute of Mathematical Sciences, IV Cross Road, CIT Campus, Tharamani, Chennai-600113, India}.
%{\small{\tt{E-mail: niru.mallick@gmail.com, nirupamam@imsc.res.ac.in}}}.\\

%  
%  \par
%  

%\noindent
%\textsc{K. Sumesh}: 
%{\small\itshape Indian Institute of Technology Madras, Sardar Patel Road, Opposite to C. L.R.I, Adyar, Chennai, -600036, India.}
%{\small{\tt{E-mail: sumeshkpl@gmail.com, sumeshkpl@iitm.ac.in}}}. \\

\end{document}